\documentclass[11pt]{amsart} 
\usepackage{cite}
\usepackage{geometry}
\newtheorem{thm}{Theorem}[section]

\newtheorem{prop}{Proposition}[section]
\newtheorem{lem}{Lemma}[section]

\theoremstyle{definition} \newtheorem{example}{Example}[section]
\newtheorem{definition}{Definition}[section]

\newcommand{\lspan}{\operatorname{span}} \newcommand{\cH}{\mathcal{H}}
 \newcommand{\cU}{\mathcal{U}}
 
\newcommand{\cP}{\mathcal{P}} \newcommand{\Nset}{\mathbb{N}}
 \newcommand{\Rset}{\mathbb{R}}
 
 \newcommand{\Wr}{\operatorname{Wr}}
\newcommand{\ord}{\operatorname{ord}}

\begin{document}

\title{Corrigendum on the proof of completeness for exceptional Hermite polynomials}
\author{David G\'omez-Ullate}
\address{Escuela Superior de Ingenier\'ia, Universidad de C\'adiz, 11519 Puerto Real, Spain.}
\address{Departamento de F\'isica Te\'orica, Universidad Complutense de
  Madrid, 28040 Madrid, Spain.}

\author{Yves Grandati}
\address{ Laboratoire de Physique et Chimie Th\'eoriques, Universit\'e de Lorraine, 57078 Metz, Cedex 3, France.}
\author{Robert Milson}
\address{Department of Mathematics and Statistics, Dalhousie University,
  Halifax, NS, B3H 3J5, Canada.}
\email{david.gomez-ullate@uca.es, grandati@univ-metz.fr,   rmilson@dal.ca}
\begin{abstract}
Exceptional orthogonal polynomials are complete families of orthogonal polynomials that arise as eigenfunctions of a Sturm-Liouville problem. Antonio Dur\'an discovered a gap in the original proof of completeness for exceptional Hermite polynomials, that has propagated to analogous results for other exceptional families. In this paper we provide an alternative proof that follows essentially the same arguments, but provides a direct proof of the key lemma on which the completeness proof is based. This direct proof makes use of the theory of trivial monodromy potentials developed by Duistermaat and Gr\"unbaum and Oblomkov.

\end{abstract}
\maketitle

\section{Introduction and definitions}

It was recently brought to our attention by Antonio Dur\'an that there is a gap in the proof of completeness for exceptional Hermite polynomials, and by extension to other polynomial families, \cite{Duran2019proof}. This is a central result on the construction of exceptional polynomials, and thus deserves immediate attention. The first papers on exceptional polynomials dealt with particular families and used a different approach to establish completeness, but the first paper to deal with the general case of exceptional Hermite polynomials indexed by partitions \cite{gomez2013rational} contains an invalid argument to support this statement (Proposition 5.8). More specifically, Proposition 5.4 in \cite{gomez2013rational} is incorrect as currently stated. A corollary of this proposition was used in  \cite{gomez2013rational} to establish the completeness of exceptional Hermite polynomials, so the whole proof must be revised.

Many of the subsequent works on the theory of exceptional polynomials rely on these propositions to establish analogous results for the completeness of other exceptional families. These are mostly the works by Dur\'an \cite{Duran2014,Duran2014a,Duran2015Jacobi,Duran2015Admi} and the more recent works by Bonneux and Kuijlaars \cite{Bonneux2018,Bonneux2019} that rely in turn on Duran's results. Some of our own papers also contain incorrect versions of Proposition 5.4 in \cite{gomez2013rational}: e.g. corollaries 5.29 and 5.30 in \cite{garcia2016bochner}, Proposition 5.3 in \cite{gomez2016} and similar claims in \cite{gomez2018durfee}. 
The problem in the argument by dimension exhaustion lies in asserting the linear independence of the constructed linear functionals at each of the zeros of $H_\lambda$, i.e. the poles of the exceptional operator. It is not hard to build a counterxample where the linear functionals $\alpha_{i,j}$ are not independent in the case where the poles have multiplicity higher than one, \cite{Duran2019proof}.
The essence of the argument can be mantained, but the explicit expressions of the differential constraints at each of the poles given in \cite{gomez2013rational,garcia2016bochner,gomez2016} cannot be employed. Rather than that, the differential constraints that characterize the exceptional subspace must be retrieved from  the Laurent expansion around such poles. Thus, in this paper we provide a direct, alternative derivation of the necessary result to prove completeness that bypasses the problematic Proposition 5.4 in \cite{gomez2013rational}. 
This result makes use of theory of trivial monodromy potentials developed by Duistermaat and Gr\"unbaum \cite{refDG86} and extended to potentials with quadratic growth at infinity by Oblomkov \cite{oblomkov1999monodromy}. We would like to stress that the main result that allows the completeness proof to hold, which is Proposition \ref{prop:etasquared} in Section~ \ref{sec:2}, is an algebraic statement independent of the $\mathrm{L}^2$ theory, and thus applies to all partitions $\lambda$, and not only to the subset of even partitions, those for which the exceptional Hermite polynomials are complete.

As mentioned above, an  independent  proof of Proposition~\ref{prop:etasquared} has been given by Dur\'an in \cite{Duran2019proof}, using the duality between exceptional Hermite and Charlier polynomials.

%
%

\subsection{Preliminaries}

We begin by recalling some necessary definitions.  A partition
$\lambda_i\; i\in \Nset$ is a non-increasing sequence of integers with
\[ \lambda_1 \geq \lambda_2 \geq \cdots \geq \lambda_\ell >0, \] and
$\lambda_i=0$ for all $i>\ell$.  Every  partition determines a
strictly decreasing sequence of integers
\[ k_i = \lambda_i + \ell-i,\quad i=1,2,\ldots \] with $k_\ell>0$,and
$k_{\ell+j} = -j,\; j=1,2,\ldots$ .  Set
\[ K_\lambda =  \{ k_1,\ldots, k_\ell \},\]
so that
$\lambda \mapsto K_\lambda$ describes a bijection
between the set of partitions and the set of finite subsets of
$\Nset$. An \textit{even partition} is a partition of length $\ell=2m$ that satisfies
\begin{equation}\label{eq:double}
 \lambda_{2i-1}=\lambda_{2i},\qquad i=1,\dots,m.
\end{equation}

Let
\[ H_n(x) = (-1)^n e^{x^2}\left( \frac{d}{dx}\right)^n e^{-x^2},\quad
  n\in \Nset_0 \] be the classical Hermite polynomials.  For a given
partition $\lambda$, let
\begin{align}
  H_\lambda &= \Wr[H_{k_\ell},\ldots, H_{k_1}], \label{eq:Hlam}\\
  H_{\lambda,i} &= \Wr[H_{k_\ell},\ldots, H_{k_1},H_i],\quad i\in \Nset_0\setminus K_\lambda \label{eq:Hlami}
\end{align}
  be Wronskians of Hermite polynomials and let
  \begin{equation}\label{eq:cU}
  \cU_\lambda = \lspan \{ H_{\lambda,i} \colon i\in \Nset_0\setminus K_\lambda \} 
\end{equation}
be the  polynomial subspace spanned by $H_{\lambda,i}$. When $\lambda$ is an even partition, $H_{\lambda,i}$ are precisely the exceptional Hermite polynomials associated to $\lambda$. We recall that $H_\lambda$ has no real zeros if and only if $\lambda$ is an even partition. For the number of real zeros that $H_\lambda$ has for other partitions, see \cite{Garcia-Ferrero2015}.

For linearly independent polynomials $p_0,\ldots, p_{n-1}$ we have
\[ \deg \Wr[p_0,p_1,\ldots, p_{n-1}] = \sum_{i=0}^{n-1} (\deg
  p_i-i).\] From this it follows that
\begin{align*}
  \deg H_\lambda &= |\lambda| = \sum_i \lambda_i,\\
  \deg H_{\lambda,i} &= |\lambda| +i- \ell.
\end{align*}
Let $I_\lambda = \{ \deg y \colon y\in \cU_\lambda\setminus\{ 0\}\}$ be
the degree sequence of $\cU_\lambda$.  By inspection,
\begin{align*}
  I_\lambda
  &= \{ \deg H_{\lambda,i} : i\notin K_\lambda \} \\
  &= \{ |\lambda|+i-\ell : i\in \Nset_0\setminus K_\lambda \} \\
  &=
    \{  n\in \Nset_0 \colon n\geq |\lambda|-\ell,\;
    n+\ell-|\lambda| \notin K_\lambda    \}
\end{align*}
Hence,
\[ \Nset_0\setminus I_\lambda =  \{ n\in \Nset_0: n< |\lambda|-\ell \}
  \cup \{ k_1+|\lambda|-\ell,\ldots, k_\ell + |\lambda|-\ell \}.
\]
It follows that the subspace $\cU_\lambda$ has codimension $|\lambda|$ in the set of all univariate polynomials $\mathbb{C}[x]$.

For any meromorphic function $f$, we define $\ord_\xi f$ to be the order of $\xi$ as a pole or root of $f$, i.e. $\ord_{\xi}f=1$ if $f$ has a simple root at $\xi$ and $\ord_{\xi}f=-2$ if $\xi$ is a double pole of $f$.

\section{Completeness of exceptional Hermite polynomials}\label{sec:2}

In this section we will first characterize the exceptional subspace $\cU_\lambda$ defined in \eqref{eq:cU} as the set of polynomials that satisfy a number of differential constraints at the zeros of $H_\lambda$. This is the part that  fixes Proposition 5.4 in \cite{gomez2013rational}. It is noteworthy that Proposition 5.4 in \cite{gomez2013rational} would hold if all roots of $H_\lambda$ are simple\footnote{The root of $H_\lambda$ at zero plays a special role due to symmetry and generally has higher multiplicity, but in the case of even partitions (the only ones that matters for the $\mathrm{L}^2$ theory), $H_\lambda$ has no real roots.}, which is precisely known as Veselov's conjecture \cite{Felder2012a}. In order to have a characterization of $\cU_\lambda$ that does not rely on any assumption on the multiplicity of the roots of $H_\lambda$, we provide Propositions~\ref{prop:Udef}  and \ref{prop:etasquared} below. We need two auxiliary Lemmas and Proposition~\ref{prop:etasquared} to prove Theorem~\ref{prop:density} on the completeness of exceptional Hermite polynomials.

\subsection{Characterization of the exceptional subspace}

The direct proof of Proposition~\ref{prop:etasquared} relies on the following results of Duistermaat-Gr\"unbaum \cite{refDG86} and Oblomkov \cite{oblomkov1999monodromy}, that we recall now without proof.

\begin{definition}\label{def:ratext}
  A Schr\"odinger operator $L=-D_{xx}+U(x)$ has trivial monodromy at $\xi\in\mathbb C$ if the general solution of the equation
  \[L[\psi]=-\psi''+U\psi = \lambda\psi\]
  is meromorphic in a neighbourhood of $\xi$ for all values of $\lambda\in\mathbb C$.
  If $L$ has trivial monodromy at every point $\xi\in\mathbb C$ we say that $L$ is monodromy-free.
\end{definition}

\begin{prop}[Duistermaat-Gr\"unbaum]
  \label{prop:dg}
  Let $U(x)$ be meromorphic in a neighbourhood of $x=\xi$ with Laurent
  expansion
  \[ U(x) = \sum_{j\geq -2} c_j (x-\xi)^j,\quad c_{-2}\neq 0.\]

  Then
  the Schr\"odinger operator $L=-D_{xx} + U(x)$ has trivial monodromy
  at $x=\xi$ if and only if there exists an integer $\nu \geq 1$ such
  that
  \begin{equation}
    \label{eq:ccond}
    c_{-2} = \nu (\nu+1),\qquad c_{2j-1} = 0,\quad 0\leq j \leq
    \nu. 
  \end{equation}

  Moreover, if the monodromy is trivial at $x=\xi$, then
  the formal eigenfunctions  
  \[L[\psi]= -\psi_{xx}+U(x)\psi(x) = E \psi(x) \]
  have  the following Laurent expansion
  \begin{equation}
    \label{eq:psilaurent}
     \psi(x) = \sum_{j=0}^\infty  a_k (x-\xi)^{j-\nu},
  \end{equation}
  with
  \begin{equation}
    \label{eq:acond}
     a_1=a_3=\ldots =a_{2\nu-1} = 0.
  \end{equation}
\end{prop}

\begin{prop}[Oblomkov]\label{prop:Obl}
  For a given partition $\lambda$ of length $\ell$, set
  \[ U_\lambda(x) = x^2 - 2 D_{xx} \log H_\lambda(x)+2\ell.\]
  Then, the Schr\"odinger operator $L_\lambda=-D_{xx} + U_\lambda(x)$ has trivial monodromy.
\end{prop}

We also recall the following result \footnote{We have chosen to perform a $\lambda$ dependent shift in the potential $\cU_\lambda$ in order to have a simpler expession for the eigenvalues, which are now independent of $\lambda$. This is a different convention than the one adopted in Proposition 5.2 of \cite{gomez2013rational}, but a purely cosmetic one.}, which follows easily from
Proposition 5.2 in \cite{gomez2013rational}.
\begin{prop}
  Let $\lambda$ be a partition of length $\ell$. Then,
  for each $k\notin K_\lambda$, the meromorphic function
  \begin{equation}
    \label{eq:psilamk}
    \psi_{\lambda,k}(x) = \frac{H_{\lambda,k}(x)}{H_\lambda(x)}
    e^{-x^2/2} 
  \end{equation}
 is a formal eigenfunction of $L_\lambda$ with
  eigenvalue $2k+1$. 
  \end{prop}

\begin{prop}\label{prop:tri}
 The  multiplicity of every root of  $H_\lambda$ is a triangular number.
\end{prop}
\begin{proof}
  This follows from Proposition \ref{prop:dg} and \ref{prop:Obl}.  The multiplicity $m_i$ of a root $\xi_i$ of
$H_\lambda$ is precisely $\frac12 c_{-2}$, where $c_{-2}$ is the leading
  coefficient of the Laurent expansion of $U_\lambda$ about the root $\xi_i$.  The desired conclusion is then a consequence of  \eqref{eq:ccond}.
\end{proof}

We have now all the necessary elements to characterize $\cU_\lambda$  as those polynomials that satisfy a set of linear differential constraints evaluated at the zeros of $H_\lambda$.

\begin{prop}\label{prop:Udef}
For an arbitrary partition $\lambda$, let $H_\lambda$ and $\cU_\lambda$ be as in \eqref{eq:Hlam} and \eqref{eq:cU}. Let $\{\xi_i\}_{i=1}^N$ be the roots of $H_\lambda$ with multiplicities $m_i=\frac{1}{2}\nu_i(\nu_i+1)$.  A polynomial $p\in\cU_\lambda$ if and only if the following linear differential constraints hold:
\begin{equation}\label{eq:diffcon}
(pF_i)^{(j)}(\xi_i)=0,\qquad j\in M_i,\quad i=1,\dots,N,
\end{equation}
where
 \begin{equation}
    \label{eq:Fidef}
     F_i(x) = \prod_{j\neq i}  (x-\xi_j)^{-m_j}    e^{-x^2/2}  
  \end{equation}
  and
\[ M_i = \{ j\in \Nset_0: j< \frac12 \nu_i(\nu_i-1) \}\cup \{ 2j-1 +
    \frac12 \nu_i(\nu_i-1) : j=1,\ldots,\nu_i \}\] 
\end{prop}

\begin{proof}
Since conditions \eqref{eq:diffcon} are linear in $p$ and independent of $k$, it suffices to prove them for a basis element $p=H_{\lambda,k}$ and they will extend by linearity to every $p\in\cU_\lambda$.
Let \[ H_\lambda(x) =\prod_{i=1}^N (x-\xi_i)^{m_i} ,\]
  where by Proposition~\ref{prop:tri} we have $m_i = \frac12 \nu_i (\nu_i+1)$ with $\nu_i\in \Nset$.
Let us fix a root $\xi_i$ and multiply \eqref{eq:psilamk} by $(x-\xi_i)^{m_i}$ to obtain
\[ (x-\xi)^{m_i} \psi_{\lambda,k}(x) = H_{\lambda,k}(x) F_i(x).\] 
Since $\psi_{\lambda,k}(x)$ is the eigenfunction of a trivial monodromy potential, then its Laurent expansion around $\xi_i$ is given by \eqref{eq:psilaurent}, and therefore we have the following Laurent expansion of  $H_{\lambda,k}(x) F_i(x)$
\[  H_{\lambda,k}(x) F_i(x)= (x-\xi_i)^{\frac12\nu_i(\nu_i-1)}\sum_{j=0} a_j (x-\xi_i)^j   \]
where by \eqref{eq:acond} we also have $ a_1=a_3=\ldots =a_{2\nu_i-1} = 0$.
The differential conditions \eqref{eq:diffcon} are precisely the ones that guarantee this behaviour around $x=\xi_i$. Note that there are always $\nu_i$ constraints coming from the vanishing of the $a_j$ coefficients, but if $\nu_i>1$ there are also an extra $\frac12\nu_i(\nu_i-1)$ set of constraints for the lower orders of the expansion. In total, at every root $\xi_i$ we have 
\[  \frac12\nu_i(\nu_i-1) + \nu_i=  \frac12\nu_i(\nu_i+1)=m_i\]
constraints, i.e. as many as the root multiplicity $m_i$.
The total number of constraints at all roots is thus
 \[ m_1+\cdots + m_N = \deg H_\lambda = |\lambda |. \]
It is easy to show that conditions \eqref{eq:diffcon} are differential expressions of different orders evaluated at each point, and they must be linearly independent since
\[  (pF_i)^{(j)}(\xi_i)=  F_i(\xi_i) \cdot p^{(j)}(\xi_i)+\text{lower order terms},\qquad j\in M_i,\quad i=1,\dots,N,  \]
and by construction $F_i(\xi_i)\neq 0$. Differential constraints evaluated at different points in the complex plane are also independent. For a proof of this last statement, see Proposition 4.12 in \cite{garcia2016bochner}.

As mentioned before, by linearity these conditions must be satisfied by any $p\in\cU_\lambda$. Since $\cU_\lambda$ has codimension $|\lambda |$ in the set of all polynomials $\mathbb{C}[x]$ by dimension exhaustion we conclude that $p\in\cU_\lambda$ if and only if conditions \eqref{eq:diffcon} hold, thus establishing the claim.

\end{proof}

\begin{prop}
  \label{prop:etasquared}
For every partition $\lambda$ and every polynomial $p\in\mathbb{C}[x]$, we have 
\begin{equation}
 H_\lambda^2 p \in \cU_\lambda
\end{equation}
\end{prop}
\begin{proof}
It suffices to show that for any polynomial $p$ we have that
  \[ (p H_\lambda^2 F_i)^{(j)}(\xi_i) = 0,\qquad j\in M_i ,\quad i=1,\ldots, N,\]
  and it follows from Proposition~\ref{prop:Udef} that $p\,H_\lambda^2\in\cU_\lambda$.
  The above claim holds because the highest order derivative is
  \[ \max M_i = 2\nu_i -1 + \frac12 \nu_i(\nu_i-1) = \frac12
    \nu_i(\nu_i+3) -1 ,\quad i=1,\ldots, N \]
  but the order of $\xi_i$ as a root of $(p H_\lambda^2 F_i)$ satisfies
  \begin{align*}
    \ord_{\xi_i} (pH_\lambda^2 F_i)\geq \nu_i(\nu_i+1),
  \end{align*}
so
 \[ \ord_{\xi_i} (pH_\lambda^2 F_i)- \max M_i\geq\frac12\nu_i(\nu_i-1)+1 >0 \]
\end{proof}

\begin{example}
The counterexample to Proposition 5.4 in \cite{gomez2013rational} found by Dur\'an in \cite{Duran2019proof} is given by $\lambda=(2,1)$, for which $H_\lambda=32x^3$. This is the simplest example for which $H_\lambda$ has a single root of multiplicity higher than one. We derive the differential constraints that define $\cU_\lambda$ by noting that
\[N=1,\quad \xi_1=0,\quad \nu_1=2, \quad M_1=\{0\}\cup\{2,4\},\quad F_1(x)={\rm e}^{-x^2 /2}.\]
The differential constraints on an element $p\in\cU_\lambda$ are:
\[ \Big(p {\rm e}^{-x^2 /2}\Big)^{(j)}(0)=0,\qquad j=0,2,4,\]
so we can characterize $\cU_\lambda$ as the following codimension 3 subspace of $\cP$:
\[ \cU_\lambda=\{p\in\cP: p(0)=p''(0)=p^{(iv)}(0)=0\}. \]
This is readily checked by noting that the polynomials \[H_{\lambda,k}=\Wr[H_1,H_3,H_k],\qquad k\in \Nset_0\setminus\{1,3\}\] do not contain any terms in $x^j$ for $j=0,2,4$. Moreover, it is also obvious that $p H_\lambda^2 \in\cU_\lambda$ for any polynomial $p$.
\end{example}

\subsection{Completeness of exceptional Hermite polynomials}

In this last section, we finish establishing the completeness of $\cU_\lambda$, following essentially the same steps given in \cite{gomez2013rational}. We restrict from here on to even partitions, in order to guarantee that  $H_\lambda>0$ for all $x\in\mathbb R$. Likewise, the variable $x$ that denoted above an arbitrary point in $\mathbb C$ will henceforth be restricted to the real line $\mathbb R$.

We define the orthogonality weight for exceptional Hermite polynomials as 

\begin{equation}\label{eq:Wlam}
W_\lambda(x) = \frac{e^{-x^2}}{H_\lambda(x)^2} dx,\quad x\in \Rset
\end{equation}
 which is ensured to be regular and possess finite moments of all orders.

Let $\cH_\alpha=\mathrm{L}^2[(0,\infty),y^\alpha e^{-y} dy]$ denote the Hilbert
space of the classical Laguerre polynomials and $\cH =
\mathrm{L}^2[\Rset,e^{-x^2} dx]$ the Hilbert space of the Hermite
polynomials. Throughout the proof, we will make use of Theorem 5.7.1 in \cite{Sz}, which asserts that $\cP$, the vector space of univariate polynomials, is
dense in $\cH_\alpha,\; \alpha>-1$ and in $\cH$.  We will write
\[  q(x) \cP(x) = \{  q(x) p(x) \colon p \in \cP \}\]
to denote  a polynomial subspace with a common factor $q(x)$.
\begin{lem}
  \label{lem:density1}
  The polynomial subspace $(1+y^m)\cP(y)$ is dense in $\cH_\alpha$ for
  every integer $m>0$ and every real $\alpha>0$.
\end{lem}
\begin{proof}
  Given a polynomial $q(y)$ and an $\epsilon>0$ it suffices to find a
  polynomial $p(y)$ such that
  \[ \Vert q- (1+y^m)p \Vert_{\cH_\alpha}^2 = \int_0^\infty \left(
    q(y) - (1+y^m) p(y)\right)^2 y^\alpha e^{-y} dy \leq
  \epsilon. \] Define the function
  \[ \hat{q}(y) =
  \begin{cases}
    \displaystyle \frac{q(y-1)}{(y-1)^m+1}, &  y\geq 1\\[8pt]
    0 & 0\leq y<1
  \end{cases}
  \]
  We assert that $\hat{q} \in \cH_{2m+\alpha}$ by observing that
  \begin{align*}
    \int_0^\infty \hat{q}(y)^2y^{2m+\alpha} e^{-y}\, dy &=
    \int_1^\infty
    \left(\frac{q(y-1)}{(y-1)^m+1}\right)^2 y^{2m+\alpha} e^{-y}\,  dy \\
    & = e^{-1} \int_0^\infty q(y)^2
    \left(\frac{(1+y)^{m}}{1+y^m}\right)^2 (1+y)^\alpha e^{-y} dy\\
    &< \infty
  \end{align*}
  Next, choose a polynomial $\hat{p}(y)$ such that
  \[ \Vert \hat{q} - \hat{p} \Vert_{\cH_{2m+\alpha}}^2 \leq
  \frac{\epsilon}{e}, \] and set
  \[ p(y) = \hat{p}(y+1).\] It follows that
  \begin{align*}
    &\int_0^\infty \left(q(y)- (1+y^m) p(y)\right)^2 y^\alpha
    e^{-y} dy \\
    &\quad = \int_0^\infty \left(\frac{q(y)}{(1+y^m)}-
      p(y)\right)^2 (1+y^m)^2 y^\alpha e^{-y} dy\\
    &\quad \leq \int_0^\infty \left(\frac{q(y)}{1+y^m}-
      p(y)\right)^2 (1+y)^{2m} y^\alpha e^{-y} dy\\
    &\quad = e\int_1^\infty \left(\hat{q}(y)- \hat{p}(y)\right)^2
    y^{2m+\alpha} \left(1-\frac{1}{y}\right)^\alpha
    e^{-y} dy\\
    &\quad \leq e\int_1^\infty \left(\hat{q}(y)- \hat{p}(y)\right)^2
    y^{2m+\alpha}
    e^{-y} dy\\
    &\quad \leq e \Vert \hat{q} - \hat{p}\Vert_{\cH_{2n+\alpha}}\\
    &\quad \leq \epsilon,
  \end{align*}
  as was to be shown.
\end{proof}

\begin{lem}
  \label{lem:density2}
  The polynomial subspace $(1+x^{2m}) \cP(x)$ is dense in $\cH$ for
  every integer $m>0$.
\end{lem}
\begin{proof}
  Given a polynomial $q(x)$ and an $\epsilon>0$ it suffices to find a
  polynomial $p(x)$ such that
  \begin{equation}
    \label{eq:d2qpepsilon}
    \Vert q- (1+x^{2m})p \Vert_{\cH}^2 = \int_\Rset
    \left( q(x) - (1+x^{2m}) p(x)\right)^2 e^{-x^2} dx \leq
    \epsilon.   \end{equation}  
  Write
  \[ q(x) = q_0 + x q_1(x^2) + x^2 q_2(x^2) \] where $q_0$ is a
  constant and where $q_1(y), q_2(y)$ are polynomials in $y=x^2$.
  Imposing the condition that $p(0)=q(0)$ let us write
  \[ p(x) = q_0+ x p_1(x^2)+ x^2 p_2(x^2) \] where $p_1(y), p_2(y)$
  are polynomials. Then, by the orthogonality of odd and even
  functions in $\cH$, the inequality \eqref{eq:d2qpepsilon} assumes
  the form
  \begin{align*}
    &\int_\Rset
    \left( q_1(x^2) - (1+x^{2m})p_1(x^2)\right)^2 x^2e^{-x^2} dx \\
    &\qquad + \int_\Rset \left(q_2(x^2) - q_0 x^{2(m-1)}-
      (1+x^{2m})p_2(x^2)\right)^2 x^4 e^{-x^2} dx\\
    &= \int_0^\infty (q_1(y) - (1+y^m) p_1(y))^2 y^{\frac{1}{2}} e^{-y} dy\\
    &\qquad + \int_0^\infty \left(q_2(y) - q_0 y^{m-1} -
      (1+y^m) p_2(y)\right)^2 y^{\frac{3}{2}} e^{-y} dy \\
    &\leq \epsilon,
  \end{align*}
  where for the first equality we employ the change of variables
  $y=x^2$.  By Lemma \ref{lem:density1}, it is possible to find
  polynomials $p_1(y), p_2(y)$ such that the above inequality is
  satisfied.
  \end{proof}

Now we are finally ready to state and prove the completeness of exceptional Hermite polynomials.

\begin{thm}\label{prop:density}
  If $\lambda$ is an even partition, the polynomial subspace $\cU_\lambda$ is dense in the Hilbert space
  $\mathrm{L}^2(\Rset,W_\lambda)$.
\end{thm}

\begin{proof}
Let $f\in \mathrm{L}^2(\Rset, W_\lambda)$.  
Set
\[ \hat{f}(x) = \frac{(1+x^{2m})}{H_\lambda(x)^2}f(x)
\] and observe that
\begin{align*}
 \int_\Rset \hat{f}(x)^2 e^{-x^2}\, dx 
  &\leq A^2 \int_\Rset \left(\frac{f(x)}{H_\lambda(x)}\right)^2
  e^{-x^2}\,dx = \int_\Rset f(x)^2 W_\lambda(x) \,dx < \infty
\end{align*}
where 
\[ A = \sup \left\{ \frac{1+x^{2m}}{H_\lambda(x)} \colon x \in \Rset
\right\}<\infty .\] Let $\epsilon>0$ be given.  
Set
\[ B = \sup \left\{ \frac{H_\lambda(x)}{1+x^{2m}} \colon x \in \Rset
\right\}<\infty.\] By Lemma \ref{lem:density2} we can find a
polynomial $p(x)$ such that
\[ \int_\Rset \left( \hat{f}(x)- (1+x^{2m}) p(x)\right)^2 e^{-x^2}\, dx
\leq \frac{\epsilon}{B^2}.\] Hence
\begin{align*}
  & \int_\Rset \left( f(x) - H_\lambda(x)^2 p(x) \right)^2 W_\lambda(x)\, dx\\
  &\quad = \int_\Rset \left( \frac{f(x)}{H_\lambda(x)} - H_\lambda(x)
    p(x) \right)^2 e^{-x^2}\, dx\\
  &\quad \leq B^2 \int_\Rset \left( \hat{f}(x) - (1+x^{2m})
    p(x) \right)^2 e^{-x^2}\, dx\\
  &\quad \leq \epsilon
\end{align*}
Finally, Proposition~\ref{prop:etasquared} ensures that the polynomial
$H_\lambda(x)^2 p(x)$ belongs to $\cU_\lambda$, which establishes the claim.
\end{proof}

\section*{Acknowledgements}
We would like to thank Antonio Dur\'an for bringing this observation to our attention and taking part in rigorous discussions on the matter. We would like to thank Paul Nevai for motivating this exchange with Antonio Dur\'an and  giving us the opportunity to prepare this amended proof.
The research of DGU has been supported in part by Spanish MINECO-FEDER Grants
MTM2015-65888-C4-3 and PGC2018-096504-B-C33. The research of RM was supported in part by NSERC grant RGPIN-228057-2009.

\end{document}